\documentclass[final,3p,times]{elsarticle}
\usepackage{amsmath, amssymb, amsthm, braket}
\newtheorem{theorem}{Theorem}[section]
\newtheorem*{theorem*}{Theorem}
\newtheorem{lemma}[theorem]{Lemma}
\newtheorem{corollary}[theorem]{Corollary}
\newtheorem*{conjecture*}{Conjecture}

\journal{Discrete Mathematics}

\begin{document}

\begin{frontmatter}



\title{Eigenvalues of neutral networks: interpolating between hypercubes}


\author[1,2]{T. Reeves}
\author[1,3]{R. S. Farr}
\author[1,4]{J. Blundell}
\author[1]{A. Gallagher}
\author[1,5]{T. M. A. Fink}

\address[1]{London Institute for Mathematical Sciences, 35a South St, London W1K 2XF, UK}
\address[2]{Department of Mathematics, Princeton University, Princeton, NJ, USA}
\address[3]{Unilever R\&D, Colworth Science Park, MK44 1LQ, Bedford, UK}
\address[4]{Department of Physics, Stanford University, Stanford, CA, USA}
\address[5]{Centre National de la Recherche Scientifique, Paris 75248, France}

\begin{abstract}
A neutral network is a subgraph of a Hamming graph, and its principal eigenvalue determines its robustness: the ability of a population evolving on it to withstand errors.
Here we consider the most robust small neutral networks: the graphs that interpolate pointwise between hypercube graphs of consecutive dimension
(the point, line, line and point in the square, square, square and point in the cube, and so on).
We prove that the principal eigenvalue of the adjacency matrix of these graphs is bounded by the logarithm of the number of vertices, 
and we conjecture an analogous result for Hamming graphs of alphabet size greater than two.
\end{abstract}


\end{frontmatter}


\noindent
The eigenvalues of neutral networks---subgraphs of Hamming graphs---is a fascinating subject, yet one which seems to have received little attention from the mathematics community.
A recent surge of scientific interest has been motivated by advances in the theory of neutral evolution \cite{wagner, draghi},
in which the evolution of a mutating population is captured by spectral properties of its underlying neutral network \cite{van N}.

A \emph{genome} is the set of all genotypes, or $a$-ary strings, of length $d$ and alphabet size $a$. 
Typically $a$ is small: $a=2$ (hydrophilic and hydrophobic), $a=4$ (nucleic acids) or $a=20$ (amino acids).
On the other hand, $d$ can range from 3 (codons) to $10^8$ (chromosomes).
We represent the genome by a $d$-dimensional {\it Hamming graph} $H_{d,a} \equiv (K_a)^d \equiv (K_a \Box \dots \Box K_a)$, where $K_a$ is the complete graph on $a$ vertices and $\Box$ is the Cartesian product \cite{Sabidussi}. Each of the $a^d$ vertices in the Hamming graph corresponds to a genotype, and two vertices share an edge if the genotypes differ by a single mutation (Hamming distance one).
A \emph{neutral network} is the set of genotypes with the same phenotype (observable characteristics); it is just a subgraph of $H_{d,a}$.
In this Note we use neutral network and phenotype interchangably.

\begin{figure}[b!]
\begin{center}
\includegraphics[width=\columnwidth]{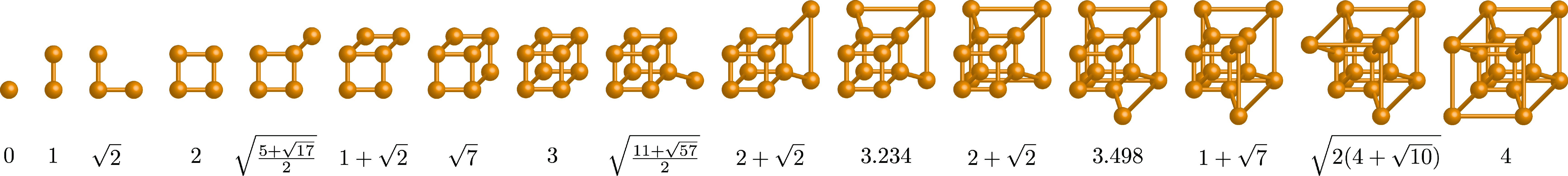}
\caption{\label{builders_clusters}
The first 16 ``bricklayer's graphs," which are the Hamming graphs and the interpolations between them. Below each is the principal eigenvalue of its adjacency matrix.
}
\end{center}
\end{figure}

Assuming the phenotype has achieved high fitness, adjoining phenotypes will have a relatively negligible growth rate and act as effective absorbing boundaries.
Now consider a mutating population on the neutral network. That portion which mutates off of it will be lost, whereas that portion which stays on will survive.
The \emph{robustness} $r$ of a neutral network is the long-term probability that a randomly selected individual mutating in a random direction survives.
It is the principal eigenvalue $\lambda$ of the adjacency matrix of the neutral network divided by the number of directions for mutation: $r = \lambda/(d(a-1))$.
This Note is about the maximum robustness that a neutral network can have, 
which for small neutral networks are themselves Hamming graphs and interpolations between them, shown for $a=2$ in Figure 1. 
Surprisingly, this ceases to be true for larger neutral networks on longer sequences ($d=19$ and above), which we demonstrate by way of counterexamples in the form of star graphs.



This Note is divided into six parts. 
In the first part, we derive the relation between the robustness $r$ of a neutral network and its principal eigenvalue $\lambda$: $r = \lambda/(d(a-1))$.
In the second part, we present our main result, $\lambda_n \leq \log_2 n$, where $\lambda_n$ is the eigenvalue of the $n$th bricklayer's graph (see Figure 1).
We prove it for all values of $n$ apart from $n=2^d\pm1$ using a geometric staircase argument.
In the third part, we set the stage for the remainder of the proof by constructing polynomials with eigenvalues $\lambda_{2^d \pm 1}$ as roots.
In the fourth and fifth parts, we bound the eigenvalues $\lambda_{2^d-1}$ and $\lambda_{2^d+1}$.
In the sixth part we conjecture a generalisation of our main result for higher $a$ and conclude with some extensions. 

A secondary motivation of this paper is to introduce mathematicians to the connection between neutral networks and spectral graph theory and to encourage them to extend this work.

\section{Relation between the robustness and the principal eigenvalue}
\noindent
In this section we derive the relation $r = \lambda/(d(a-1))$, first given in \cite{van N}. First, we define the phenotype robustness $r$ as a weighted average over genotype robustness; second, we define it as the extent to which mutation off the neutral network does not deplete the growth due to fitness.
Readers who are less interested in the biological motivation may skip this section.

\emph{Genotype and phenotype robustness.}
Consider a neutral network $P$, and let its adjacency matrix be $A$. The genotype robustness $r_i$ \cite{wagner} of a genotype $g_i$ is the probability of a mutation being neutral: the number of neutral edges incident to $g_i$ (i.e. edges which do not lead to a different phenotype)
divided by the total number of incident edges $d(a-1)$. The genotype robustness can therefore be written as 
\begin{align}
\label{geno}
r_i = \frac{\sum_j A_{ij}}{d(a-1)}.
\end{align}
For a neutral network, let $n(t)$ be its population vector at time $t$, with the $i$th component $n_i(t)$ corresponding to the population on genotype $g_i$. The normalized population is distributed according to $n(t)/\sum_i n_i(t)$. Suppose for now that in the limit $t\to\infty$, the normalized population is distributed according to a unique distribution. Then we define the phenotype robustness $r$ to be the long-time population-weighted average of the genotype robustnesses $r_i$:
\begin{align}
\label{pheno}
r = \frac{\sum_i n_i(\infty) r_i}{\sum_i n_i(\infty)}.
\end{align}
It is the fraction of the population flux that is neutral. We will now determine this limit.

\emph{Mutational flux and fitness.} Mutation induces a population flux across neighbouring genotypes. If the mutation rate per letter is $\mu$, the mutational flux is $1-(1-\mu)^{d(a-1)} \approx \mu d(a-1)$ for $\mu d(a-1) \ll 1$. It is the fraction of a population that mutates per generation. 
Some of this mutational flux will also cross phenotypic boundaries when neighbouring genotypes lie in two different phenotypes. That which does not cross phenotypic boundaries is neutral. The fitness $f$ is the raw reproductive rate of the phenotype. After $t$ generations, the total population of a neutral network will have changed by a factor of $f^t$, in the absence of mutations.


{\it Mutation matrix}. 
The action of mutation on the population distribution over a single generation can be expressed by the mutation matrix $M$: 
\begin{align}
M = (1-\mu d(a-1)) I + \mu A.
\label{K}
\end{align}
The first term is the probability that no mutation occurs and the second the probability of mutating.
Being symmetric, $A$ can be diagonalised by an orthonormal set of eigenvectors $x_i$:
\begin{align}
M = (1-\mu d(a-1))\sum_i x_i x_i^\intercal + \mu \sum_i x_i \nu_i x_i^\intercal,
\label{mutation_matrix}
\end{align}
where the $x_i$ satisfy the eigenvalue equation $A x_i = \nu_i x_i$. The population vector $n(t)$ is obtained by transforming an initial vector $n_0$ by $M^t$ and multiplying it by $f^t$:
\begin{align}
\label{population_distribution}
n(t) = f^t M^t n(0) = \sum_i x_i^\intercal n(0) f^t \left(1-\mu d(a-1) \left(1-\frac{\nu_i}{d(a-1)}\right)\right)^t x_i.
\end{align}
Let $\nu_1$ be the largest (principal) eigenvalue of $A$, denoted hereafter $\lambda$.
Since $\lvert \nu_i \rvert \leq d$,  all terms $i>1$ decay exponentially with respect to the first for $\mu>0$. 
In the large time limit the sum is dominated by the first term, whose eigenvalue $\nu_1 \equiv \lambda$ is largest:
\begin{align}
n_t \approx x_1^\intercal n(0) f^t \left(1-\mu d(a-1) \left(1-\frac{\lambda}{d(a-1)}\right)\right)^t x_1.
\label{n_t_large_time_limit}
\end{align}
We now show that defining the robustness as $\lambda/(d(a-1))$ agrees with the definition of phenotype robustness in (\ref{pheno}). Indeed, plugging (\ref{geno}) into (\ref{pheno}),
\begin{align*}
r &= \frac{\sum_{ij} A_{ij}n_i(\infty)}{d(a-1)\sum_i n_i(\infty)} \\
&= \frac{\sum_{j}\lambda n_j(\infty)}{d(a-1)\sum_i n_i(\infty)} \\
&= \frac{\lambda}{d(a-1)}.
\end{align*}

The quantity $r$ therefore measures how well the shape of the neutral network can reduce the rate of deleterious mutation acting on the population as a whole. We see from (\ref{n_t_large_time_limit}) that at large time $t$ at every generation, a fraction $\mu d(a-1)(1-r)$ of the population mutates off the neutral network, and the growth rate $(1 - \mu d (1-r)) f$ is the fitness that can be usefully employed to increase the population and not spent replenishing population lost to deleterious mutations incurred at the boundary. The steady state distribution of the population depends only on the shape of the neutral network and on neither the mutation rate $\mu$ nor the fitness $f$.

\section{Neutral networks with large eigenvalues}
\noindent
\emph{Bricklayer's graphs.}
Just how robust a phenotype can be---or how large an eigenvalue a neutral network can have---has remained an open question. 
For short sequences ($d \leq 4, a=2$), we found from exhaustive enumeration that the most robust neutral networks are themselves hypercubes  
or interpolations between them, illustrated in Figure 1.
Computational sampling for slightly longer sequences ($5 \le d \le 9, a=2$) agrees with this. 
We generalize the sequence of graphs and interpolations between them in Figure 1 for $a>2$ as follows: 
suppose all vertices $\{q\}$ in $H_{d,a}$ are labelled as integers from 0 to $a^d-1$, and two vertices share an edge if their base $a$ representations differ in exactly one digit.
Then $G_{n,a}$ is the subgraph induced by the vertices $q < n$. 
We call these graphs $G_{n,a}$ ``bricklayer's graphs'' because they form the sequence by which a bricklayer would instinctively fill in the Hamming graph $H_{d,a}$.
For the remainder of this Note we set the alphabet size $a=2$, so we are only concerned with hypercubes and their subgraphs.
For simplicity we denote $G_{n,2}$ by $G_n$.
We conjecture an extension of our main result for general $a$ in the Conclusion.

For neutral networks on sequences of short length $d$, the bricklayer's graphs $G_n$ are the most robust; they have maximal principal eigenvalues.
For longer lengths $d$, however, a surprise is in store: the $G_n$ are \emph{not} the most robust neutral networks. 
In particular, we discovered the following counterexamples for $d=19$ and above.
Let $S_n$ be the star graph: a tree with one internal vertex and $n$ leaf vertices. 
The principal eigenvalue of $S_n$ is readily found to be $\sqrt{n}$. 
Now let us compare the eigenvalue of a star of $n$ vertices, $S_{n-1}$, to the eigenvalue $\lambda_n$ of $G_n$.
As we prove below, $\lambda_n \leq \log_2 n$.
For the bricklayer's graphs $G_n$ to win, we need $\sqrt{n-1} < \lambda_n$, implying 
\begin{align*}
\sqrt{n-1} <  \log_2 n.
\end{align*}
However, this not true for $n \ge 20$. It is an open question as to what shape does maximize the eigenvalue for larger graphs. In the concluding remarks of \cite{friedman}, the authors consider the possibility that Hamming balls (graphs consisting of all points that are most a given distance from a point) are asymptotic maximizers of hypercube subgraphs, but then provide some evidence that they are not.

\emph{Our main result.}
In this Note we prove that the principal eigenvalue $\lambda_{n}$ of the bricklayer's graph $G_{n}$ satisfies $\lambda_n \leq \log_2 n$.
Our general approach is to show by a geometric staircase argument that for $d \ge 3$, a slightly stronger inequality ($\lambda_n < \log_2 (n-1)$) holds for most $n$; it will then suffice to examine the cases where $n = 2^d \pm 1$, using polynomials that have $\lambda_{2^d \pm 1}$ as roots.

\begin{theorem*}
\label{bricklayerclustertheorem}
For all graphs $G_n$, we have $\lambda_n \leq \log_2 n$, with equality if and only if $n$ is a power of 2.
\end{theorem*}

Equality is attained in the Theorem if $n$ is a power of $2$ since $\lambda_n$ must lie between the mean and maximum vertex degree \cite{Brouwer}, and for $n$ a power of 2, all vertices are of degree $\log_2 n$.
We wish to show that if $n$ is not a power of $2$, there is strict inequality. We make two observations. Observation 1: Since the principal eigenvalue of a proper subgraph of a connected graph is less than the principal eigenvalue of the graph itself, it follows that if $n < m$ then $\lambda_n < \lambda_m$. Observation 2: Since $G_{2n} = G_{n} \Box K_2$, and the spectrum of a Cartesian product of graphs is the sum of their individual spectra \cite{Brouwer}, it follows that if $\lambda_n < \log_2 n$ then $\lambda_{2n} = \lambda_n + 1 < \log_2 2n$. Using these observations, we claim:

\begin{lemma}\label{reducetoedgecases} 
The Theorem is true for all $n$ if for some $k$,
\begin{align}
\label{strongercondition}
\lambda_n < \log_2 (n-1) \textnormal{ for $2^k + 2 \le n \le 2^{k+1} - 1$},
\end{align} 
and also
\begin{align}
\label{minusapointtheorem}
&\textnormal{$\lambda_{2^d-1} < \log_2 (2^d-2)$, $d \geq 5$, and} \\
\label{plusapointtheorem}
&\textnormal{$\lambda_{2^d+1} < \log_2 (2^d+\tfrac12)$, $d \geq 3$}.
\end{align}
\end{lemma}

\begin{proof}
The ``staircase" argument is illustrated in Figure \ref{staircase_figure}. We verify numerically that the Theorem is true for $n \leq 16$ and (\ref{strongercondition}) holds for $k=3$. Now if (\ref{strongercondition}) is true for some $k$, then by Observation 2,
\begin{align}
\label{overlappinginequalities}
\lambda_{2n}  < \log_2 (2n-2) \textnormal{ for $2^k + 2 \leq n \leq 2^{k+1} - 1$}.
\end{align}
By Observation 1 we have the expansion 
\begin{align}
\label{overlappinginequalitiesexpansion}
\lambda_{2n-2} < \lambda_{2n-1} < \lambda_{2n}  < \log_2 (2n-2) < \log_2 (2n-1) < \log_2 2n
\end{align}
for $2^k + 2 \leq n \leq 2^{k+1} - 1$, so that $\lambda_{m}  < \log_2 m$ for $2^{k+1} + 2 \leq m \leq 2^{k+2} - 2$. Conditions (\ref{minusapointtheorem}) and (\ref{plusapointtheorem}) then ensure that $\lambda_{m}  < \log_2 m$ for $m = 2^{k+1} + 1$ and $m=2^{k+2} - 1$ as well, proving the Theorem for $2^{k+1} \le n \le 2^{k+2}$.
Finally, note that (\ref{overlappinginequalitiesexpansion}) together with (\ref{minusapointtheorem}) and (\ref{plusapointtheorem}) imply that (\ref{strongercondition}) holds with $k$ replaced by $k+1$, so we may repeat our induction indefinitely.
\end{proof}

\begin{figure}[b!]
\begin{center}
\includegraphics[width=0.6\columnwidth]{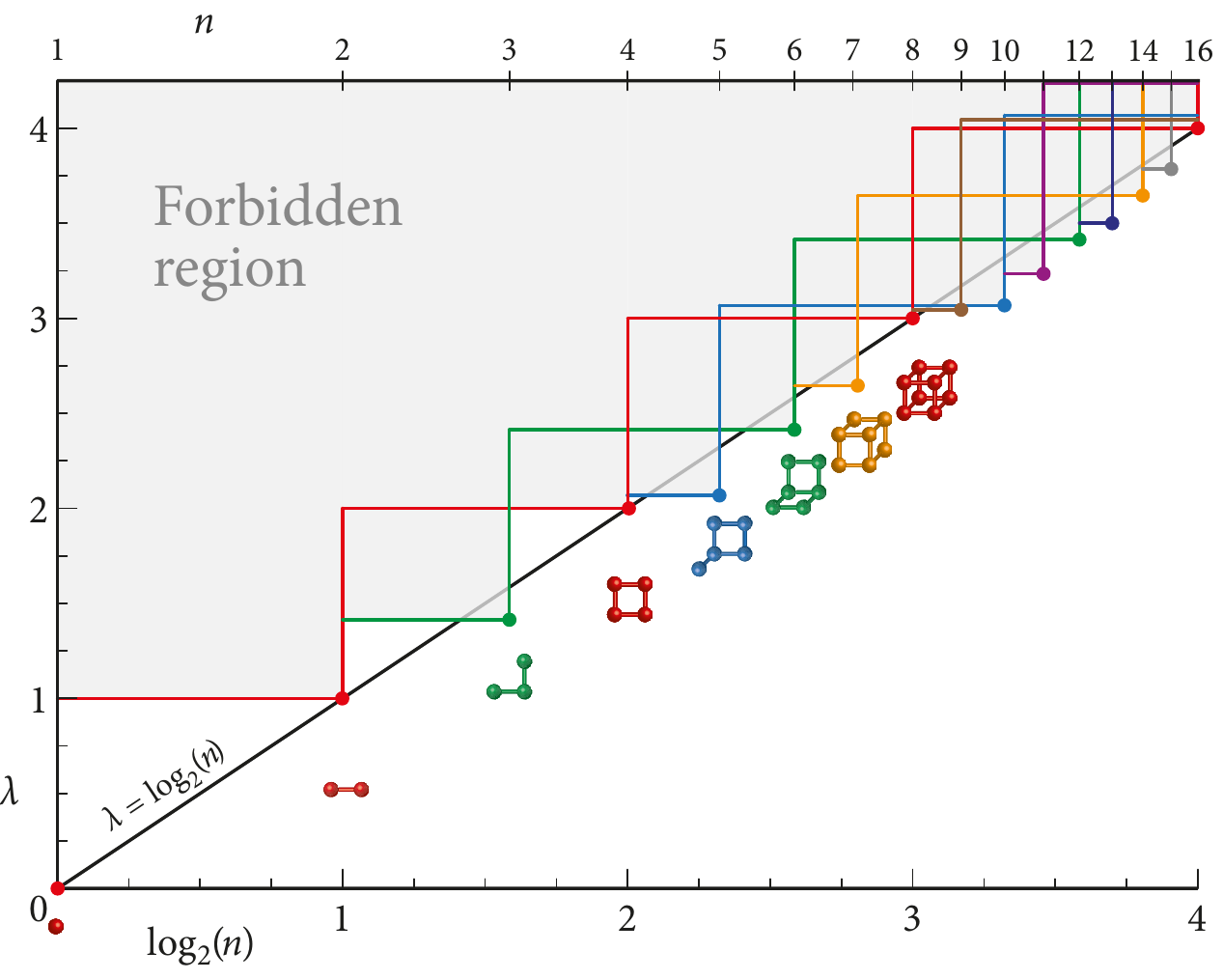}
\caption{\label{staircase_figure}
The ``staircase" argument for the proof of Lemma \ref{reducetoedgecases}. Because each bricklayer's graph is a subgraph of its successor, knowing the principal eigenvalue of a bricklayer's graph immediately places restrictions on higher-dimensional bricklayer's graphs (as demonstrated by the staircase figures of each color).
}
\end{center}
\end{figure}

Therefore, the Theorem reduces to (\ref{minusapointtheorem}) and (\ref{plusapointtheorem}), which we will prove by looking at polynomials that have the eigenvalues of our desired graphs as roots.
We will make use of the following standard theorem in linear algebra:

$\newline$
(Cauchy's Interlacing Theorem \cite{Hwang}).
\textit{Let $A$ be an $n \times n$ symmetric nonnegative matrix with eigenvalues $a_1 \leq \dots \leq a_n$, and let $B$ be an $m \times m$ principal submatrix of $A$ with eigenvalues $b_1 \leq \dots \leq b_m$. Then for all $j < m+1$, $a_j \leq b_j \leq a_{n-m+j}$.}

\section{Polynomials with eigenvalues $\lambda_{2^d \pm 1}$ as roots}
\noindent
Let $\chi_{n}$ be the characteristic polynomial of the adjacency matrix of $G_{n}$. Enumerating the hypercube spectrum, we find
\[
\chi_{2^d} (x) = \prod_{i=0}^d (x - (d - 2i))^{\binom{d}{i}},
\]
so we may define the polynomial
\begin{align}
\label{P2ddefinition}
P_{2^d} (x) = \frac{\chi_{2^d} (x)}{\prod_{i=1}^{d-1} (x - (d - 2i))^{\binom{d}{i}-1}} = \prod_{i=0}^d (x - (d - 2i)).
\end{align}

By applying Cauchy's Interlacing Theorem to the adjacency matrices of $G_{2^{d}+1}$ and $G_{2^d}$ (thus ``sandwiching" the spectrum of $\chi_{2^{d}+1}$ by that of $\chi_{2^{d}}$), we see from the multiplicity of the eigenvalues (of $\chi_{2^d}$) that $\chi_{2^d+1} (x)$ has the factor $\prod_{i=1}^{d-1} (x - (d - 2i))^{\binom{d}{i}-1}$. A similar argument shows that $\chi_{2^d-1} (x)$ has the same factor. So we may define the polynomials
\[
P_{2^d \pm 1} (x) = \frac{\chi_{2^{d} \pm 1} (x)}{\prod_{i=1}^{d-1} (x - (d - 2i))^{\binom{d}{i}-1}}.
\]
Furthermore, we see that $\chi_{2^d-1}$ has at most $d$ simple roots, all of which are also roots of the $d$-degree polynomial $P_{2^d-1}$.

The reason we use the polynomials $P$ is related to the fact that $G_{2^d-1}$ is a Hamming ball of radius $d-1$ in $H_{d,2}$. We define $B_{d,r}$, the $d$-dimensional ball of radius $r$, as the set of points in $H_{d,2}$ that are Hamming distance at most $r$ from the ``origin'' (the point labelled ``0'' according to the labelling scheme specified in the first paragraph of section 2). We determine recursive equations that give $\lambda$, the principal eigenvalue of $B_{d,r}$. Consider the corresponding eigenvector $w$, and note that by symmetry the component of $w$ corresponding to a given vertex depends only on the distance of the vertex from the origin. Therefore, let $w_k$ be the value of the component of $w$ corresponding to a vertex of distance $k$ from the origin. By matrix multiplication, we find that 
\begin{align}
\label{lambda1}
\lambda w_0 &= d w_1 \\
\label{lambda2}
\lambda w_k &= kw_{k-1} + (d-k)w_{k+1} ~\textnormal{for $1 \leq k < r$} \\
\label{lambda3}
\lambda w_r &= rw_{r-1}.
\end{align}
By setting $w_0 = 1$ and following the equations above for each fixed $r$, we find that the principal eigenvalue of $B_{d,r}$ is a root of the polynomial $p_r (\lambda)$, where $p_0 = \lambda$, $p_1 = \lambda^2 - d$, and
\begin{align*}
p_r &= \lambda p_{r-1} - r(d-r+1)p_{r-2} \textnormal{ for $r \geq 2$.}
\end{align*}

Applying this to $\lambda_{2^k-1}$, we can generate polynomials in $\lambda$ with coefficients in $d$ (say $f_{k}(d,\lambda)$) such that when $k$ is substituted for $d$, the resulting polynomial in $\lambda$ has $\lambda_{2^k-1}$ as a root. Then $f_{1}(d,\lambda) = \lambda$,
$f_{2}(d,\lambda) = \lambda^2 - d$ and, in general, 
\begin{align}
\label{frecursiverelation}
f_{k}(d,\lambda) = \lambda f_{k-1}(d,\lambda)- (k-1)(d-k+2)f_{k-2}(d,\lambda),
\end{align}
and $f_{k}(k,\lambda)$ has $\lambda_{2^k-1}$ as a root. In fact, $f_{k}(k,\lambda)$ has every simple eigenvalue of $G_{2^k-1}$ as a root (by the reasoning of the derivation). Since the degree of $f_{k}(k,\lambda)$ as a polynomial in $\lambda$ is $d$, and we have from above that $\chi_{2^d-1}$ has at most $d$ simple roots (all of which are also roots of the $d$-degree polynomial $P_{2^d-1}$), it must be the case that 
\begin{align}
\label{pequalsf}
P_{2^k-1}(\lambda) = f_{k}(k,\lambda).
\end{align}

\section{Bounding $\lambda_{2^d-1}$}
\noindent
Rewriting the right side of (\ref{minusapointtheorem}) by applying the Taylor expansion with Lagrange remainder gives
\begin{align*}
\log_2(2^d - 2) = d + \log_2 \left(1-\frac{2}{2^d} \right) >d + \frac{1}{\log 2} \left(-\frac{2}{2^d} - \frac{2^2}{2^{2d-1}} \right).
\end{align*}
Therefore, for $d \geq 5$,
\begin{align}
\label{minusapointlogbound}
\log_2(2^d - 2) > d - \frac32 \, \frac{2}{2^d}.
\end{align}

Now we deal with the left side of (\ref{minusapointtheorem}).
\begin{lemma}
\label{minusapointnonexplicitbound}
$\lambda_{2^d-1} < d - P_{2^d-1}(d)/P'_{2^d-1}(d)$.
\end{lemma}
\begin{proof}
Note that the function $P_{2^d-1}(x)$ is convex on $x \geq \lambda_{2^d-1}$. To see this, let $f(x) = \prod_{i=1}^n (x-r_i)$ be any monic polynomial with all real roots and observe that
\begin{align*}
f''(x)
= 2 \sum_{\substack{j_1,j_2=1 \\ j_1 < j_2}}^n \left[ \prod_{i=1,i \neq j_1, i \neq j_2}^n (x-r_i) \right],
\end{align*}
which is always nonnegative if $x$ is at least the largest root of $f$. Now since the tangent linear approximation of a convex function is an underestimate, we obtain $P_{2^d-1}(d) + P'_{2^d-1}(d) \cdot (\lambda_{2^d-1}-d) < 0$, which implies the lemma.
\end{proof}

We evaluate the desired values of $P_{2^d-1}$ and its derivative using (\ref{pequalsf}) and the recursive relation in (\ref{frecursiverelation}).
\begin{lemma}
\label{minusapointp}
$P_{2^d-1}(d) = d!$
\end{lemma}
\begin{proof}
Recall the definition of $f_k(d,\lambda)$ from the previous section and the fact that $P_{2^k-1} (\lambda) = f_k(k,\lambda)$. We seek to prove that $f_{k}(k,k) = k!$, and to do this we will prove a stronger claim that for $i, k \in \mathbb{N}^+$,
\[
f_{k}(i,i) = (i)_k,
\]
where $(i)_k = i  (i-1)  (i-2) \cdots (i-k+1)$ is the Pochhammer symbol.
We use induction on $k$. For $k = 1$ we have $f_{1}(i,i) = i$ and $f_{2}(i,i) = i^2 - i = i(i-1)$. Supposing the claim is true for $f_{k-2}(i,i)$ and $f_{k-1}(i,i)$, we find from (\ref{frecursiverelation}) that
\begin{align*}
f_{k}(i,i) &= i f_{k-1}(i,i)- (k-1)(i-k+2)f_{k-2}(i,i) \\
&= i (i)_{k-1} - (k-1)(i-k+2) (i)_{k-2} \\
&= (i)_k.
\qedhere
\end{align*}
\end{proof}

\begin{lemma}
\label{minusapointpderivative}
\[
P'_{2^d-1}(d) = d! \sum_{j=0}^{d-1} \frac{2^j}{j+1}.
\]
\end{lemma}
\begin{proof}
With respect to the polynomials $f_{k}(d,\lambda)$, let $f'_{k}(d,\lambda)$ denote $\frac{\partial }{\partial \lambda} f_{k}(d,\lambda)$. Then $P'_{2^k-1}(\lambda) = f'_{k}(k,\lambda)$, and $f'_{1}(d,\lambda) = 1$, $f'_{2}(d,\lambda) = 2\lambda$, and for $k \geq 3$, from (\ref{frecursiverelation}),
\begin{align*}
f'_{k}(d,\lambda) &= f_{k-1}(d,\lambda) + \lambda f'_{k-1}(d,\lambda) - (k-1)(d-k+2)f'_{k-2}(d,\lambda).
\end{align*}
We seek to prove that $f'_{k}(k,k) = k! \sum_{j=0}^{k-1} \frac{2^j}{j+1}$, and to do this we will prove a stronger claim that for integers $k > 0, i \ge 0$, 
\begin{align*}
f'_{k}(k+i,k+i) = k! \sum_{j=0}^{k-1} \binom{k-j+i-1}{i} \frac{2^{j}}{j+1}.
\end{align*}

We use induction on $k$. For the base cases $k = 1$ and $k = 2$ we find 
that $f'_{1}(1+i,1+i) = 1$ and $f'_{2}(2+i,2+i) = 2(2+i)$, as desired. Now supposing the claim is true for $f'_{k-2}(k+i,k+i)$ and $f'_{k-1}(k+i,k+i)$, it follows that

\begingroup
\allowdisplaybreaks

\small
\begin{align*}
f'_{k}(k+i,k+i) &= f_{k-1}(k+i,k+i) + (k+i) f'_{k-1}(k+i,k+i) - (k-1)(i+2)f'_{k-2}(k+i,k+i) \\
&= (k+i)_{k-1} + (k+i) (k-1)! \sum_{j=0}^{k-2} {\textstyle \frac{(k-j+i-1)!}{(i+1)!(k-j-2)!} \frac{2^{j}}{j+1} - (k-1)!} \sum_{j=0}^{k-3} \textstyle \frac{(k-j+i-1)!}{(i+1)!(k-j-3)!} \frac{2^{j}}{j+1} \\
&= {\textstyle \frac{(k+i)!}{(i+1)!} + (k-1) \frac{(k+i)!}{(i+1)!}} + \sum_{j=0}^{k-2} \left[ \textstyle \frac{(k-1)!}{(k-j-2)!} \frac{(k-j+i)!}{(i+1)!} + j  \frac{(k-1)!}{(k-j-2)!} \frac{(k-j+i-1)!}{(i+1)!} \right] {\textstyle \frac{2^{j}}{j+1}} -(k-1)(k-2) {\textstyle \frac{(k+i-1)!}{(i+1)!}} - \sum_{j=2}^{k-2} \textstyle \frac{(k-1)!}{(k-j-2)!} \frac{2^{j-1}}{j} \frac{(k-j+i)!}{(i+1)!} \\
&= {\textstyle \frac{(k+i)!}{(i+1)!} + (k-1) \textstyle \frac{(k+i)!}{(i+1)!} + (k-1)(k-2)\frac{(k+i-1)!}{(i+1)!}} + \sum_{j=2}^{k-2} \left[ \textstyle \frac{(k-1)!}{(k-j-2)!} \frac{2^j}{j+1} + (j-1)  \frac{(k-1)!}{(k-j-1)!} \frac{2^{j-1}}{j} \right] {\textstyle \frac{(k-j+i)!}{(i+1)!}} \\
&\phantom{=} + (k-2)(k-1)! {\textstyle \frac{2^{k-2}}{k-1} -(k-1)(k-2)\frac{(k+i-1)!}{(i+1)!}} - \sum_{j=2}^{k-2} {\textstyle \frac{(k-1)!}{(k-j-2)!} \frac{2^{j-1}}{j} \frac{(k-j+i)!}{(i+1)!} } \\
&= k {\textstyle \frac{(k+i)!}{(i+1)!}} + \sum_{j=2}^{k-2} \textstyle \left[ \textstyle \frac{(k-1)!}{(k-j-2)!} \left( \frac{2^j}{j+1} - \frac{2^{j-1}}{j} \right) + (j-1)  \frac{(k-1)!}{(k-j-1)!} \frac{2^{j-1}}{j} \right] \frac{(k-j+i)!}{(i+1)!} + [2(k-1)(k-1)! - k!] \textstyle \frac{2^{k-2}}{k-1} \\
&= k {\textstyle \frac{(k+i)!}{(i+1)!}} + \sum_{j=2}^{k-2} \textstyle \frac{k!}{(k-j-1)!} \left( \frac{2^j}{j+1} - \frac{2^{j-1}}{j} \right) \frac{(k-j+i)!}{(i+1)!} + k! \left( \frac{2^{k-1}}{k} - \frac{2^{k-2}}{k-1} \right) \\
&= k! \sum_{j=0}^{k-1} \left[ \textstyle \frac{(k-j+i)!}{(i+1)!(k-j-1)!} \frac{2^j}{j+1} - \frac{(k-j-1)(k-j+i-1)!}{(i+1)!(k-j-1)!} \frac{2^j}{j+1} \right] \\
&= k! \sum_{j=0}^{k-1} \textstyle \frac{(k-j+i-1)!}{i!(k-j-1)!} \frac{2^{j}}{j+1}.
\qedhere
\end{align*}
\normalsize

\endgroup

\end{proof}

Now substituting the results of Lemmas \ref{minusapointp} and \ref{minusapointpderivative} into Lemma \ref{minusapointnonexplicitbound},
\begin{align}
\label{minusapointbound}
\lambda_{2^d-1} < d - \frac{1}{\sum_{j=0}^{d-1} \frac{2^j}{j+1}}.
\end{align}
A simple induction with base cases $1 \le d \le 3$ shows that the sum in the denominator of the second term of (\ref{minusapointbound}) on the right-hand side can be bounded by
\begin{align}
\label{sumbound}
\sum_{j=1}^{d} \frac{2^j}{j} <  3 \, \frac{2^d}{d}.
\end{align}
Therefore,
\begin{align}
\label{minusapointsumbound}
d - \frac{1}{\sum_{j=0}^{d-1} \frac{2^j}{j+1}} < d - \frac{2}{3} \, \frac{d}{2^d}.
\end{align}
Since we may combine the bounds (\ref{minusapointsumbound}) and (\ref{minusapointlogbound}) for $d \geq 5$, we have proved (\ref{minusapointtheorem}).

\section{Bounding $\lambda_{2^d+1}$}
\noindent
Rewriting the right side of (\ref{plusapointtheorem}) in much the same way as the previous section, we find that for $d \geq 3$,
\begin{align}
\label{plusapointlogbound}
\log_2(2^d + \tfrac12) > d + \frac{\tfrac12}{2^d}.
\end{align}
As for the left side of (\ref{plusapointtheorem}), we have a computational shortcut:

\begin{lemma}
\label{Ge}
Let $G$ be a graph, and let $e$ be a bridge of $G$ (i.e., $e$ is an edge such that removing it would increase the number of connected components of $G$). Let $G^{*}$ be the graph $G$ with $e$ removed, and $G^{**}$ be the graph $G$ with $e$ and its endpoints removed. 
Then $\chi_{G} = \chi_{G^{*}} - \chi_{G^{**}}$, where $\chi_{G}$ is the characteristic polynomial of the adjacency matrix of $G$.
\end{lemma}
\begin{proof}
This is an extension of Lemma 1 in \cite{Lovasz}, which states the result when $G$ is a forest and $e$ is any edge, and is also Theorem 1.3 in \cite{stevanovic}. The proof involves expanding the matrix whose determinant is $\chi_{G}$, using Laplacian expansion and linearity of the determinant.
\end{proof} 

\begin{corollary}
\label{prelation}
$\chi_{2^d+1}(\lambda) = \lambda \chi_{2^d}(\lambda) - \chi_{2^d-1}(\lambda)$, and so dividing by $\prod_{i=1}^{d-1} (x - (d - 2i))^{\binom{d}{i}-1}$,
\[
P_{2^d+1}(\lambda) = \lambda P_{2^d}(\lambda) - P_{2^d-1}(\lambda).
\]
\end{corollary}

\begin{lemma}
\label{plusapointnonexplicitbound}
$\lambda_{2^d+1} < d - P_{2^d+1}(d)/P'_{2^d+1}(d)$.
\end{lemma}
\begin{proof}
From the preceding corollary, 
\begin{align*}
P''_{2^d+1}(x) = 2P'_{2^d}(x) + x P''_{2^d}(x) - P''_{2^d-1}(x).
\end{align*}
We wish to show that this is nonnegative on $x \geq d$. From the argument of the proof of Lemma \ref{minusapointnonexplicitbound}, $P'_{2^d}(x) \geq 0$ for $x \geq d$, and from the equation for the second derivative of a polynomial there, $P''_{2^d}(x) > P''_{2^d-1}(x)$ for $x \geq d$ since the roots of $P_{2^d-1}$ interlace those of $P_{2^d}$ by Cauchy's Interlacing Theorem. So $P_{2^d+1}(x)$ is convex on $x \geq d$, and so by linear approximation, $P_{2^d+1}(d) + P'_{2^d+1}(d) \cdot (\lambda(G_{2^d+1})-d) < 0$, which implies the lemma.
\end{proof}

Now as in the previous section, we evaluate the desired values of $P_{2^d+1}(d)$ and its derivative.
\begin{lemma}
\label{plusapointp}
$P_{2^d+1}(d) = -d!$.
\end{lemma}
\begin{proof}
Using Corollary \ref{prelation} and Lemma \ref{minusapointp}, $P_{2^d+1}(d) = d p_{2^d}(d) - p_{2^d-1}(d) = 0 - d!$.
\end{proof}

\begin{lemma}
\label{plusapointpderivative}
\[
P'_{2^d+1}(d) = d! \left( d \, 2^d - \sum_{j=0}^{d-1} \frac{2^j}{j+1} \right).
\]
\end{lemma}
\begin{proof}
Differentiating (\ref{P2ddefinition}) and then plugging in $d$ gives $P'_{2^d}(d) = 2^d d!$.
Then using Corollary \ref{prelation} and Lemma \ref{minusapointpderivative},
\begin{align*}
P'_{2^d+1}(d) &= P_{2^d}(d) + d  P'_{2^d}(d) - P'_{2^d-1}(d) \\
&= 0 + d  \, 2^d d! - d! \sum_{j=0}^{d-1} \frac{2^j}{j+1} \\
&= d! \left( d \, 2^d - \sum_{j=0}^{d-1} \frac{2^j}{j+1} \right).
\qedhere
\end{align*}
\end{proof}

Substituting the results of Lemmas \ref{plusapointp} and \ref{plusapointpderivative} into Lemma \ref{plusapointnonexplicitbound},
\begin{align}
\label{plusapointbound}
\lambda_{2^d+1} < d + \frac{1}{d \, 2^d - \sum_{j=0}^{d-1} \frac{2^j}{j+1}}.
\end{align}
From (\ref{sumbound}), we obtain that for $d > 1$,
\begin{align}
\label{plusapointsumbound}
d + \frac{1}{d \, 2^d - \sum_{j=0}^{d-1} \frac{2^j}{j+1}} < d + \frac{1}{(d - \frac{3}{2d})2^d}.
\end{align}
Since we may combine the bounds (\ref{plusapointsumbound}) and (\ref{plusapointlogbound}) for $d \geq 5$, we have proved (\ref{plusapointtheorem}). 

This concludes the proof of the Theorem. 

\section{Conclusion}
\noindent
As an additional remark, if we bound $\log_2 (2^d -N)$ for general $N$ in the manner of (\ref{minusapointlogbound}), we can deduce an asymptotic result: for any $N$, there exists $D$ so large that $\lambda_{2^d-1} < \log_2 (2^d - N)$ for all $d > D$. Similarly generalizing (\ref{plusapointlogbound}) leads to the result that for any $\epsilon > 0$, there exists $D$ so large that $\lambda_{2^d+1} < \log_2 (2^d + \epsilon)$ for all $d > D$.

Throughout most of this paper, we have set $a=2$.
We conjecture an extension of the Theorem for general $a$:

\begin{conjecture*}
For all graphs $G_{n,a}$, we have $\lambda_{n,a} \le (a-1) \log_a n$, with equality if and only if $n$ is a power of $a$.
\end{conjecture*}

There are several interesting and potentially important questions that we have not considered here, which merit further investigation.
We prove that the form of the hypercube subgraph with maximal eigenvalue is a bricklayer's graph for small $d$ but the general form of the maximizers is unknown. 
Indeed it is an open avenue of study to find even non-trivial bounds on the eigenvalue of a hypercube subgraph in terms of its number of vertices. 
While for small dimension $d$ the bricklayer's graphs are optimal, for $s \ge 20$ and $d \ge s-1$, Hamming balls of radius 1 are superior. 
For large $s$ and $d\ll s$, Hamming balls of larger radius may eventually dominate, but this is unproven. 
How these transitions extend to larger values of alphabet size $a$ is also an open question, though it seems that the critical dimension separating bricklayer's graphs and balls grows with $a$. 
We hope that further research by others will shed light on these questions.






\end{document}